\documentclass[12pt,twoside,a4paper,fleqn]{article}
\usepackage{euscript,a4,times}
\usepackage[cp866]{inputenc}
\usepackage[english]{babel}
\usepackage{makeidx}
\usepackage{latexsym,amsfonts,amssymb,amsmath,longtable,amsthm}
\catcode`@=11
\newbox\tr@tto
\setbox\tr@tto=\hbox{{\count0=0\dimen0=-,9pt\dimen1=1,1pt\loop\ifnum\count0<11 \advance \count0 by1 \vrule width.51pt height\dimen1
     depth\dimen0\kern-0.17pt\advance\dimen0 by-0.05pt\advance\dimen1
     by0.1pt\repeat \loop\ifnum\count0<21\advance \count0 by1 \vrule
     width.6pt height\dimen1 depth\dimen0\kern-0.2pt \advance\dimen0
     by-0.1pt\advance\dimen1 by 0.05pt\repeat}}
\def\medint{\displaystyle\copy\tr@tto\kern-10.4pt\int}
\catcode`@=12

\input amssym.def

\def\Xint#1{\mathchoice
   {\XXint\displaystyle\textstyle{#1}}%
   {\XXint\textstyle\scriptstyle{#1}}%
   {\XXint\scriptstyle\scriptscriptstyle{#1}}%
   {\XXint\scriptscriptstyle\scriptscriptstyle{#1}}%
   \!\int}
\def\XXint#1#2#3{{\setbox0=\hbox{$#1{#2#3}{\int}$}
     \vcenter{\hbox{$#2#3$}}\kern-.5\wd0}}

\def\dashint{\Xint-}
\newcommand{\R}{{\mathbb R}}
\newcommand{\N}{{\mathbb N}}

\newcommand{\Le}{{\mathcal L}}

\renewcommand{\H}{{\mathcal H}}
\renewcommand{\L}{{\mathcal L}}

\newcommand{\supp}{\mathop{\rm supp}}
\newcommand{\Int}{\mathop{\rm Int}\nolimits}
\newcommand{\Cl}{\mathop{\rm Cl}\nolimits}
\newcommand{\Capp}{\mathop{\rm Cap}\nolimits}
\newcommand{\Var}{\mathop{\rm Var}}

\def\ve{\varepsilon}
\newcommand{\loc}{{\rm loc}}

\newcommand{\diam}{\mathop{\rm diam}}

\newcommand{\LL}{\mathrm{L}}
\newcommand{\WW}{\mathrm{W}}

\newcommand{\CC}{\mathrm{C}}
\newcommand{\BV}{\mathrm{BV}}

\newtheorem{lem}{Lemma}[section]
\newtheorem{ttt}[lem]{Theorem}
\newtheorem{cor}[lem]{Corollary}
\theoremstyle{definition}

\newtheorem{df}[lem]{Definition}

\newcommand{\dd}{{\rm d}}

\title{ON THE MORSE-SARD PROPERTY AND LEVEL SETS OF SOBOLEV AND BV FUNCTIONS}
\author{Jean Bourgain, Mikhail V.~Korobkov\footnote{The author was supported by the Russian
Foundation for Basic Research (project no. 08-01-00531-a)and by
Federal Target Grant ``Scientific and educational personnel of
innovation Russia'' for 2009-2013 (government contract No.
02.740.11.0457).} \, and Jan Kristensen\footnote{Work supported by the EPSRC Science 
and Innovation award to the Oxford Centre for Nonlinear PDE (EP/E035027/1)}}

\date{}

\begin{document}

\maketitle

%\begin{center} {\large\bf
%ON MORSE-SARD PROPERTY AND LEVEL SETS OF SOBOLEV AND BV FUNCTIONS}

%\medskip
%J.~Bourgain, Institute for Advanced Study in Princeton, USA,
%e-mail: {\it bourgain@ias.edu}
%
%M.V.~Korobkov\footnote{The author was supported by the Russian
%Foundation for Basic Research (project no. 08-01-00531-a)and by
%Federal Target Grant ``Scientific and educational personnel of
%innovation Russia'' for 2009-2013 (government contract No.
%02.740.11.0457).}, Sobolev Institute of mathematics, Novosibirsk,
%Russia, e-mail: {\it korob@math.nsc.ru}
%
%J.~Kristensen, Mathematical Institute, University of Oxford, 
%e-mail: {\it kristens@maths.ox.ac.uk}
%\end{center}

\begin{abstract}
We establish Luzin $N$ and Morse--Sard properties for $\BV_2$-functions defined on 
open domains in the plane. Using these results we prove that almost all level sets
are finite disjoint unions of Lipschitz arcs whose tangent vectors
are of bounded variation. In the case of $\WW^{2,1}$--functions we strengthen the
conclusion and show that almost all level sets are finite disjoint unions
of $\CC^1$--arcs whose tangent vectors are absolutely continuous.

\medskip
\noindent 
{\bf Key words:} {\it $\BV^2$ and $\WW^{2,1}$--functions, Luzin $N$--property,
Morse--Sard property, level sets.}
\end{abstract}

\section*{Introduction}

For $\CC^2$--smooth functions $v\colon \Omega\to\R$, defined on an open 
subset $\Omega$ of $\R^2$, the classical Morse--Sard theorem \cite{Mo}, \cite{S} 
(see also~\cite{Fed} or \cite{Hi}) guarantees that
\begin{equation}
\label{in1} \Le^1(v(Z_v))=0,
\end{equation}
where $\Le^1$ is the 1--dimensional Lebesgue measure on $\R$ and $Z_v$ is
the critical set of $v$, $Z_v=\{x\in\Omega\,:\,\nabla v(x)=0\}$. Whitney
demonstrated~\cite{Wh} that the $\CC^2$--smoothness condition in the
above assertion cannot be dropped. Namely, he constructed a
$\CC^1$--smooth function $v\colon (0, 1)^2\to\R$ for which the set $Z_v$ 
of critical points contains an arc on which $v$ is not constant (subsequently
called a Whitney arc).

However, some analogs of Sard's theorem are valid for the
functions lacking the required smoothness in the classical theorem. Although~(\ref{in1})
may be no longer valid then, A.~Ya.~Dubovitski\u{i}~\cite{Du}
obtained some results on the structure of level sets in the case
of reduced smoothness (also see~\cite{B}).

Another Sard--type theorem was obtained by A.V.~Pogorelov (see
\cite[Chapter~9, Section~4]{Po}): For a function $v\in
\CC^1(\Omega)$ on a plane domain $\Omega$, the equality~(\ref{in1}) holds if 
for any linear map $L\colon \R^2\to\R$ the sum $v(x)+L(x)$ satisfies the 
maximum principle (see also~\cite{Ko06} for another proof of this result). In
particular, the equality~(\ref{in1}) holds if the gradient range
$\nabla v(\Omega)$ has no interior points (see also
\cite{Ko06,Ko09,K07}).

Another direction of the research was the generalization of Sard's
theorem to functions in H\"{o}lder and Sobolev spaces (for example,
see~\cite{B,DeP,Fig,KK05,Nor}). In particular, De~Pascale (see also \cite{Fig}) 
proved that~(\ref{in1}) holds when $v\in
\WW^{2,p}_\loc(\Omega)$ for $p>2$. Note that in this case $v$ is
$\CC^1$--smooth by virtue of the Sobolev imbedding theorem, and so the critical
set is defined as usual.

In the paper~\cite{Bu} it was proved that for functions $v\in
\WW^{2,p}_\loc(\R^2)$ with $p>1$ there are no Whitney arcs.

Landis \cite{La} proved that the equality~(\ref{in1}) holds if
$v\colon \Omega\to\R$ is a difference of two convex functions (sometimes
called a d.c.-function), a result which answered a question raised previously
by A.V.~Pogorelov. D.~Pavlica and L.~Zaj\'{i}\v{c}ek \cite{PZ}
presented the detailed and modern proof of the Landis result.
Moreover, they proved in~\cite{PZ} that the equality~(\ref{in1})
holds for Lipschitz functions $v\in \BV_{2,\loc}(\Omega)$, where
$\BV_{2,\loc}(\Omega)$ is the space of functions $v\in
\LL^{1}_\loc(\Omega)$ such that all its partial (distributional)
derivatives of the second order are $\R$-valued Radon measures on
$\Omega$.

In this paper we extend the last result to the case of any $\BV_{2}$--function 
defined on a planar domain (without the additional Lipschitz assumption,
see Theorem~\ref{Th1}). Moreover, as we understand the critical set
in a wider sense than in~\cite{PZ}, our result is also an improvement
in the Lipschitz case. More precisely, in~\cite{PZ} the critical set is 
defined as the set of points~$x$, where $v$ is (Frechet--)differentiable 
with total (Frechet--)differential $v'(x)=0$. But it
is known \cite{Dor} (see also Lemma~\ref{lem3.1.1} below) that in
general a function $v\in \BV_{2,\loc}(\Omega)$ admits a continuous representative
which is differentiable outside an at most $\H^1$-$\sigma$-finite
(rectifiable) set, and that has ``half-space differentials''
$\H^1$-almost everywhere. We include in the critical set $Z_v$ the
points $x\in\Omega$ such that one of the ``half-space differentials''
is zero at $x$.

Our main result, contained in Theorem \ref{Thh3.3} and Corollary \ref{Thh3.1},
is to establish the Luzin $N$--property with respect to $\H^1$ for $\BV^2$ 
functions on plane domains.
More precisely, we show that if $v$ is $\BV^2$ on the open domain $\Omega \subset \R^2$, 
then for any $\varepsilon >0$ there exists $\delta >0$ such that for all
subsets $E\subset\Omega$ with $\H_\infty^1(E) < \delta$ we have 
$\L^1(v(E))< \varepsilon$. In particular, it follows
that $\L^1(v(E))=0$ whenever $\H^1(E)=0$.  So the image of the exceptional ``bad'' set, 
where neither the differential nor the half-space differentials are defined, has zero Lebesgue
measure. This ties nicely in with our definition of the critical set and our
version of the Morse--Sard result for $\BV^2$--functions on the plane.

Finally, using these results we prove that almost all level sets of $\BV_2$--functions
defined on open domains in the plane, are finite disjoint unions of Lipschitz arcs whose 
tangent vectors have bounded variations (Theorem~\ref{Th3.1} and Corollary~\ref{cor3.2}). 
In the $\WW^{2,1}$--case we can strengthen the conclusions and show that almost all level 
sets are finite disjoint unions of $\CC^1$--arcs whose tangent vectors are absolutely 
continuous functions (Theorem~\ref{Th2.1} and Corollary~\ref{cor2.2}).

After this work was completed we learned that \cite{ACK} have also recently established the 
Morse--Sard property for $\WW^{2,1}$ functions on the plane.

\section{Preliminaries}

Throughout this paper $\Omega$ denotes an open subset of $\R^2$.
By {\it a domain} we mean an open connected set. For a general subset $E \subset \R^2$, 
we let $\Cl{E}$ stand for its closure, and $\partial{E}$ for its boundary.

For a distribution $T$ on $\Omega$ denote by $D_iT$,
$i=1,2$, the distributional partial derivatives of $T$, and write $DT = (D_{1}T,D_{2}T)$. 
For $\R$-valued and $\R^2$--valued Radon measures $\mu$ we denote by $\|\mu\|$ the total
variation measure of $\mu$. The space $\BV (\Omega)$ is as usual defined as consisting
of those functions $f\in \LL^1(\Omega)$ whose distributional
partial derivatives $D_if$ are Radon measures with
$\|D_if\|(\Omega)<\infty$ (for detailed definitions
see~\cite{EG}). As a consequence of Radon--Nikodym's theorem we have for any $f\in
\BV(\Omega)$ the polar decomposition of the distributional
derivative $Df(E) = \int_{E} \! \nu \, \dd \| Df \|$, where
$\nu \colon \Omega \to {\mathbb S}^1$ is a Borel vector field
valued in the unit sphere ${\mathbb S}^1 \subset \R^2$, and $\| Df \|$
is the total variation measure of $Df$.

A central role is played by $\BV^2(\Omega)$ defined as the space of functions $v\in
\LL^1(\Omega)$ such that $D_iv\in \BV(\Omega)$,
$i=1,2$. It is known (see~\cite{M}) that each function $v\in
\BV_2(\Omega)$ has a continuous representative, and subsequently we shall always select this
representative when discussing $\BV^2$--functions. For $v\in
\BV_2(\Omega)$ denote by $\nabla v$ the gradient mapping $\nabla
v = (D_{1}v,D_{2}v)\colon \Omega\to\R^2$, well--defined as a $\BV (\Omega , \R^2 )$
mapping. Denote also
$$\|v\|_{\BV_2(\Omega)}=
\|v\|_{\LL^1(\Omega)}+\|\nabla v\|_{\LL^1(\Omega)}+\|D^2v\|(\Omega),$$
$$
\WW^{1,1}(\Omega)=\{f\in \LL^1(\Omega): \, D_if\in \LL^1(\Omega),
i=1,2 \},
$$
$$
\WW^{2,1}(\Omega)=\{v\in \LL^1(\Omega): \, D_if\in \WW^{1,1}(\Omega),
i=1,2 \}.
$$
We write $\|v\|_{\BV}$ instead of $\|v\|_{\BV(\R^2)}$.

For a Lebesgue measurable set $F\subset \R^2$ and a point
$x\in\R^2$ we use the following notation:
$$
\bar D(F,x)=\limsup\limits_{r\to0+}\frac{\mathcal L^2(F\cap B(x,r))}{\mathcal
L^2 (B(x,r))},\quad \underline
D(F,x)=\liminf\limits_{r\to0+}\frac{\mathcal L^2(F\cap
B(x,r))}{\mathcal L^2 (B(x,r))},
$$
$$
\Int_MF=\{x\,:\,\underline
D(F,x)=1\},\quad \Cl_MF=\{x\,:\,\bar D(F,x)>0\},
$$ 
$$
\partial^MF= \Cl_M F\setminus\Int_MF.
$$
Here $\mathcal L^2$ is the Lebesgue measure on $\R^2$.
Denote by $\H^1$, $\H^{1}_{\infty}$ the 1-dimensional Hausdorff measure, Hausdorff content, respectively:
for any $F \subset \R^2$,
$\H^1(F)=\lim\limits_{\alpha\to0+}\H^1_\alpha(F),$ where
$$
\H^1_\alpha(F)=\inf\bigl\{
\sum_{i=1}^\infty\diam F_i\ :\ \diam F_i\le\alpha,\ \ F
\subset\bigcup\limits_{i=1}^\infty F_i\bigr\}.
$$
Recall that for any function $f\in \BV (U)$, where $U$ is an open
set in $\R^2$, the coarea formula
$$
\|Df\|(U)=\int\limits_{-\infty}^{+\infty}\H^{1}\bigl(U\cap\partial^M\{f\le\lambda\}\bigr)\,
\dd \lambda
$$
holds (see~\cite{EG}).

\section{On images of sets of small capacities under $\BV_2$ functions on the plane.}

The main result of this section is the following Luzin $N$--property for 
$\BV^2$--functions:

\begin{ttt}
\label{Thh3.3} {\sl Let $v\in \BV_{2}(\R^2)$. Then for any
$\varepsilon>0$ there exists $\delta>0$ such that for any set $E\subset\R^2$ \ if \
$\H^1_\infty(E)<\delta$ then
$\H^1(v(E))<\varepsilon$. }
\end{ttt}

\begin{cor}
\label{Thh3.1} {\sl If $v\in \BV_2(\R^2)$, $E\subset\R^2$, and
$\H^1(E)=0$, then $\H^1 (v(E))=0$. }
\end{cor}
Fix a function $v\in \BV_{2}(\R^2)$. To prove the above results we
need some preliminary lemmas.

\begin{lem}
\label{lem3.2} {\sl For any $\varepsilon>0$ there exists
$\delta>0$ such that for any set $E\subset\R^2$ if $\H^1_\infty(E)<\delta$ then
$\|D^2v\|(E)<\varepsilon$. }
\end{lem}

\begin{proof}
This is a consequence of the Coarea formula.
\end{proof}

\begin{lem}
\label{lemc1} {\sl For each $f\in \BV(\R^2)$ and for any
$\varepsilon_0>0$ there exists a pair of functions $f_0,f_1\in
\BV(\R^2)$ such that
\begin{eqnarray}
\label{c1} f = f_0 +f_1 ,\\
\label{c2} \|f_0\|_{\LL^{\infty}}\le K ,\\
\label{c3} \|f_1\|_{\BV}<\varepsilon_0,
\end{eqnarray}
where $K=K(\varepsilon_0,f)$. }
\end{lem}

\begin{proof}
The proof is similar to the proof of Theorem 3 in
\cite[\S5.9]{EG}.

Fix $K>0$ and denote
$$
f_0(x)=\left\{
\begin{array}{lcr}
f(x), & |f(x)|\le K; \\
K , & f(x)> K, \\
-K  , & f(x)<-K, \\
\end{array} \right.
$$
$$
f_1(x)=f(x)-f_0(x).
$$
Obviously $\|f_1\|_{\LL^{1}}<\frac12\varepsilon_0$ for
sufficiently large $K$. By construction we have inclusions
$f_0,f_1\in \BV(\R^2)$ (see, for example, Theorem 4(iii) in
\cite[\S4.2.2]{EG} for the Sobolev case). Then by the coarea formula
$$
\|Df_1\|(\R^2)=\int\limits_{|\lambda|>K}\H^{1}\bigl(\partial^M\{f\le\lambda\}\bigr)\,
\dd\lambda.
$$
Consequently $\|f_1\|_{\BV}<\frac12\varepsilon_0$ for
sufficiently large $K$. 
\end{proof}

\begin{cor}
\label{lemc3} {\sl For any $\varepsilon_0>0$ there exists a pair of
functions $f_0,f_1\in \BV(\R^2,\R^2)$ such that
\begin{eqnarray}
\label{c4} \forall x\in\R^2\quad\nabla v(x)\equiv f_0(x)+f_1(x);\\
\label{c5} \|f_0\|_{\LL^{\infty}}\le K;\\
\label{c6} \|f_1\|_{\BV}<\varepsilon_0.
\end{eqnarray}
}
\end{cor}
By  {\it interval } we mean a square with the sides
parallel to the coordinate axis.

\begin{lem}[see, for example, \cite{M}]
\label{lb1} {\sl Let  $I\subset\Omega$ be an interval of
the size $\ell(I)$. Then
\begin{equation}
\label{cb1} \H_1(v(I))\leq C\Big\{ \|D^2 v\|(I)+
\frac 1{\ell(I)}
\int_I |\nabla v|\Big\},
\end{equation}
where $C$ does not depend on $I,v$.
}
\end{lem}

\begin{lem}[see also~\cite{Bo}]
\label{lemkb1} {\sl Denote by $\mathcal C$ the collection of
all functions of the form
$$
\varphi=\frac 1{\H^1(\partial \Omega)} 1_\Omega,
$$
where $1_\Omega$ is the indicator function of the set $\Omega$ and
$\Omega$ is a bounded domain in $\mathbb R^2$ with a smooth
boundary $\partial \Omega$. If $f\in \BV(\R^2)$ and
\begin{equation}
\label{kb1} \|f\|_{\BV}\le1,
\end{equation}
then there exists a sequence of functions $f_i \colon\R^2\to\R$ such that
$f_i\to f$ almost everywhere, and each $f_i$ is a convex combination 
of functions from $\mathcal C\cup(-\mathcal C)$.
}
\end{lem}

\begin{proof}
We may assume without loss of generality that
\begin{equation}
\label{kb2} f\ge0,\quad \|\nabla f\|_{\LL^1}<1
\end{equation}
(see the proof of Lemma~\ref{lemc1}).
Since each function from $\BV(\R^2)$ can be approximated by functions from 
$\CC^\infty_0(\R^2)$ (see \cite[\S5.2.2]{EG}), we may also assume without loss of 
generality that
\begin{equation}
\label{kb3} f\in \CC^\infty_0(\R^2),\quad \supp f\subset B(0,R), \quad f(\R^2)\subset[0,M].
\end{equation}
For a parameter $\delta<1$ consider $f_\delta=f+g+c$, where $c$ is a constant and 
$g\colon \R^2\to\R$ is a linear function with small norm such that

(i) $\|\nabla f_\delta\|_{\LL^1(B(0,R))}<1$,

(ii) $\sup\limits_{x\in B(0,R)}|f(x)-f_\delta(x)|<\delta$,

(iii) all the critical values of the function $f_\delta$ are irrational 
numbers and they are regular in the sense of Morse theory,

(iv) for each rational $t>\delta$ we can decompose the preimage as
$$
\{x\in B(0,R): f_\delta(x)>t\}=\bigcup\limits_{i=1}^{m_t}\Omega_i ,
$$
where $\Omega_i$ are bounded smooth domains, and
$$
\Omega_i\cap\Omega_j=(\partial\Omega_i)\cap(\partial\Omega_j)=\emptyset\quad\mbox{for }i\ne j,
$$
$$
(\partial\Omega_i)\cap \partial B(0,R) =\emptyset\quad\mbox{for }i=1,\dots,m_t.
$$
Then the function $h\colon [\delta,M+1]\to\R$, defined by the formula
$$
h(t)=\H^1 \Bigl( B(0,R)\cap\{f_\delta=t\} \Bigr) ,
$$
is continuous, and hence in particular integrable in the Riemann sense. 
By (i) and by the Coarea formula we get
$$
\int_\delta^{M+1}h(t)\, \dd t<1.
$$
In view of the definition of the Riemann integral we have for sufficiently large 
$k\in\N$ that
$$
\sum\limits_{\N\ni j>k\delta}\frac1kh(t_j)<1,
$$
where $t_j=\frac{j}k$.
Write $E_j=\{x\in B(0,R):f_\delta(x)>\frac{j}k\}$ and $\tilde f_j=\frac1k{1_{E_j}}$.
By construction
\begin{equation}
\label{kb6} \|f-\sum\limits_{\N\ni j>k\delta}\tilde f_j\|_{L^\infty}<3\delta+\frac2k.
\end{equation}
Let $E_j=\bigcup\limits_{i=1}^{m_j}\Omega^i_j$,
where the $\Omega^i_j$ are defined in (iv).
By construction
\begin{equation}
\label{kb7} \sum\limits_{\N\ni j>k\delta}\sum\limits_{i=1}^{m_j}
\frac1k\H^1(\partial \Omega^i_j)=\sum\limits_{\N\ni j>k\delta}
\frac1kh(t_j)<1.
\end{equation}
Finally
\begin{equation}
\label{kb8}
\sum\limits_{\N\ni j>k\delta}\tilde f_j=
\sum\limits_{\N\ni j>k\delta}\sum\limits_{i=1}^{m_j}\alpha_{ij}
\frac{1_{\Omega^i_j}}{\H^1(\partial \Omega^i_j)},
\end{equation}
where
\begin{equation}
\label{kb9}
\alpha_{ij}=\frac{\H^1(\partial \Omega^i_j)}k,
\end{equation}
and consequently by (\ref{kb7}),
\begin{equation}
\label{kb10}
\sum\limits_{\N\ni j>k\delta}\sum\limits_{i=1}^{m_j}\alpha_{ij}<1.
\end{equation}
Formulas (\ref{kb6}),(\ref{kb8}) and (\ref{kb10}) give the required assertion. 
\end{proof}

\begin{df}
\label{db1} Let $\mu$ be a positive measure on $\mathbb R^2$. We
say that $\mu$ has property $(*)$ if $\mu$ is absolutely continuous with respect to
Lebesgue measure (so $\mu(I)=\int_Ig(x)\,\dd x$,\ \ where $g\in
\LL^1(\R^2)$) and
\begin{equation}
\label{cb2} \mu(I) \leq \ell(I)
\end{equation}
for any interval $I\subset\mathbb R^2$.
\end{df}

\begin{lem}
\label{lb2} {\sl If $f\in \BV(\R^2)$ and $\mu$ has property
$(*)$, then
\begin{equation}
\label{cb3} \Big|\int f d\mu\Big| \leq C\Vert f\Vert_{\BV},
\end{equation}
where $C$ does not depend on $\mu,\ f$. }
\end{lem}

\begin{proof}
Because of Lemma~\ref{lemkb1} and the Fatou lemma, it is sufficient to bound 
$\int\varphi \, \dd \mu$ for the functions of the form
$$
\varphi=\frac 1{\H^1(\partial \Omega)} 1_\Omega,
$$
where $\Omega$
is a bounded domain in $\mathbb R^2$ with
a smooth boundary $\partial \Omega$.
Obviously $\Omega \subset I$, where $I$ is an interval of size
$\ell(I)\sim \diam \Omega\leq \H^1(\partial \Omega)$.
Hence
$$
\int\varphi \, \dd \mu\leq \frac{\mu(I)}{\H^1(\partial \Omega)} \lesssim
\frac{\mu(I)}{\ell(I)} < C,
$$
as required.
\end{proof}

\begin{cor}
\label{lb3} {\sl
If $f\in \BV_2(\mathbb R^2)$ and $\mu$ is a measure with property
$(*)$, then
\begin{equation}
\label{cb4} \int|\nabla f| \, \dd \mu\leq C\Vert f\Vert_{\BV_2},
\end{equation}
where $C$ does not depend on $\mu,\ f$. }
\end{cor}
By {\it a dyadic interval} we
understand a square of the form
$[\frac{k}{2^m},\frac{k+1}{2^m}]\times[\frac{l}{2^m},\frac{l+1}{2^m}]$,
where $k,l,m$ are integers.

\noindent
The following assertion is straightforward, and hence we omit its proof
here.

\begin{lem}
\label{lemD} {\sl For any bounded set $F\subset\R^2$ where exist
dyadic intervals $I_1,\dots,I_4$ such that $F\subset
I_1\cup\dots\cup I_4$ and $\ell(I_1)=\dots=\ell(I_4)\le2\diam F$.}
\end{lem}

\begin{proof}[Proof of Theorem~\ref{Thh3.3}]
Fix $\varepsilon_0>0$ and take a decomposition $\nabla v=f_0+f_1$ from Lemma~\ref{lemc3}.
If $\delta$ from the conditions of Theorem~\ref{Thh3.3} is sufficiently small, we
may write
$$
E\subset\bigcup I_\alpha,
$$
where $\{I_\alpha\}$ is a collection of dyadic intervals satisfying
\begin{equation}
\label{cb8} \sum_\alpha\ell(I_\alpha)< 16\delta<
\frac{1}{K+1}\varepsilon_0
\end{equation}
(see Lemma~\ref{lemD}). Define
$$
\mathcal F= \left\{ J \, : \, J\subset\mathbb R^2 \text { dyadic interval}; \sum_{I_\alpha\subset J} 
\ell(I_\alpha)\geq \ell(J) \right\}.
$$
Thus $I_\alpha\in \mathcal F$ for each $\alpha$.
Denote by $\mathcal F^* =\{J_\beta\}$ the collection of maximal elements of $\mathcal F$.
Clearly
\begin{equation}
\label{cb9}
E\subset \bigcup_\alpha I_\alpha \subset \bigcup_\beta J_\beta ,
\end{equation}
and since dyadic intervals are either disjoint or contained in one another, the
$\{J_\beta\}$ are mutually disjoint.
It follows that
\begin{equation}
\label{cb10}
\sum_\beta\ell(J_\beta)\leq \sum_\beta\sum_{I_\alpha\subset
 J_\beta} \ell(I_\alpha) \leq \sum_\alpha\ell
(I_\alpha)\overset {(\ref{cb8})}< 16\delta<\frac{1}{K+1}\ve_0.
\end{equation}
Observe also that for any dyadic interval $Q\subset\mathbb R^2$,
\begin{equation}
\label{cb11}
\sum_{J_\beta\subset Q}\ell(J_\beta) \leq \sum_{I_\alpha\subset Q}\ell(I_\alpha)\leq 2\ell(Q).
\end{equation}
We used here that if $J_\beta\subset Q$ for some $\beta$, then either
$J_\beta=Q$ or $Q\not\in \mathcal F$ (because $J_\beta$ is maximal);
and in both cases (\ref{cb11}) holds.
Define the measure $\mu$ by
\begin{equation}
\label{cb12}
\mu = \left(\sum_\beta \frac 1{\ell(J_\beta)} 1_{J_\beta}\right) {\mathcal L}^{2}.
\end{equation}

\noindent 
{\bf Claim.}  $\frac1{48}\mu$ has property $(*)$.

\medskip

\noindent
Indeed, write for a dyadic interval $Q$,
$$
\mu(Q) =\sum_{J_\beta\subset Q} \ell(J_\beta)+\sum_{Q\subset J_\beta}\ \frac {\ell(Q)^2}{\ell(J_\beta)}
\leq 3\ell(Q),
$$
where we invoked (\ref{cb11}) and the fact that $Q\subset J_\beta$ for at most one $\beta$.
Then for any interval $I$ we have the estimate
$\mu(I)\le48\ell(I)$ (see Lemma~\ref{lemD}).
This proves the claim.

\noindent
Now return to $\H^1\big(v(E)\big)$. From (\ref{cb9}) we get
$$
v(E) \subset\bigcup_\beta v(J_\beta).
$$
Given $\varepsilon_0 > 0$ it follows from the conditions of Theorem~\ref{Thh3.3} and 
using Lemma~\ref{lem3.2} and inequality~(\ref{cb10}) that if $\delta> 0$ is
sufficiently small, then we may assume 
\begin{equation}
\label{cb13}
\sum\limits_{\beta}\|D^2v\|(J_\beta)<\varepsilon_0 ,
\end{equation}
By Lemma \ref{lb1} and properties~(\ref{c4})--(\ref{c6}), (\ref{cb4})

\begin{eqnarray*}
\sum_{\beta} \H^1(v(J_\beta)) &\leq& C\sum_{\beta} \|D^2v\|(J_\beta)+
C\sum_{\beta} \frac1{l(J_\beta)}\int_{J_\beta}|\nabla v|\\
&\leq&  C\varepsilon_0 +C\frac{K}{K+1}\varepsilon_0
+C\sum_{\beta} \frac1{l(J_\beta)}\int_{J_\beta}|f_1|\\
&=& C'\varepsilon_0+C\int |f_1| \,d\mu \leq C''\varepsilon_0.
\end{eqnarray*}
Since $\ve_0$ may be taken arbitrary small, it follows that
Theorem~\ref{Thh3.3} is proved.
\end{proof}

\section{Sard--type theorem}

Before stating the main result of this section we shall define our notion of
critical set for $v \in \BV^{2}_{\loc}(\Omega )$, where $\Omega \subset \R^2$
is open. First we let for $\varepsilon > 0$,
$$
E_\varepsilon=\{x\in\Omega\ :\ |\nabla v(x)|\le\varepsilon\},
$$
and note that $\Cl_M E_\varepsilon$ does not depend on the particular representative
we use for $\nabla v$ when defining $E_\varepsilon$. Define
$$
Z_{0v}=\bigcap\limits_{\varepsilon>0}\Cl_M E_\varepsilon ,
$$
and
$$
Z_{1v}=\{x\in\Omega\ : v \mbox{ is differentiable at  } x \mbox{ and } v'(x)=0\} ,
$$
where we refer to the continuous representative of $v$ alluded to in the introduction
(see also Lemma \ref{lem3.1.1} below).
The critical set for $v$ is the union $Z_v=Z_{0v}\cup Z_{1v}$.

\begin{ttt}
\label{Th1} {\sl Suppose $v\in \BV^{2}_{\loc}(\Omega)$, where
$\Omega$ is a domain in $\R^2$. Then $\H^1( v(Z_v))=0$.}
\end{ttt}
The proof of Theorem~\ref{Th1} splits into a number of lemmas. Further we may assume, 
without loss of generality, that $\Omega=B(0,1)\subset\R^2$ and
$v\in \BV^2(\Omega)$.

\noindent
We require the following known result about differentiability properties of $\BV^2$-functions.

\begin{lem}[see \cite{Dor}, Theorems~B and 1]
\label{lem3.1.1} {\sl
We can choose the Borel representative of
$\nabla v$ such that there exist a decomposition $\R^2=K_v\cup G_v\cup
A_v$ and mappings $\lambda\colon \R^2\to\R^2$, $\mu\colon \R^2\to\R^2$,
$\nu\colon K_v\to \mathbb{S}^1$ with the following properties:

(i) $\H^1(A_v)=0$.

(ii) $K_v=\bigcup\limits_i K_i$, each $K_i$ is a compact subset of
some $\CC^1$--curve $L_i$; moreover, $\nu(x)$ is perpendicular to
$L_i$ if $x\in K_i$.

(iii) for all $x\in G_v$, \ $\nabla v(x)=\lambda(x)=\mu(x)$ and
$$
\lim_{r\searrow 0}\dashint_{B(x,r)}|\nabla v(z)-\nabla v(x)|^2\,\dd z=0,
$$
$$
\sup_{y\in B(x,r)}r^{-1}|v(y)-v(x)-y\cdot
\nabla v(x)|\to0\quad\mbox{as }r\searrow 0
$$ 
(i.e., $v$ is differentiable at $x$);

(iv) for all $x\in K_v$,
$$
\lim_{r\searrow 0}\dashint_{B_+(x,r)} \! |\nabla v(z)-\lambda(x)|^2\,\dd z=0,
$$
$$
\lim_{r\searrow 0}\dashint_{B_-(x,r)} \! |\nabla v(z)-\mu(x)|^2\,\dd z=0,
$$
$$
\sup_{y\in B_+(x,r)}r^{-1}|v(y)-v(x)-y\cdot
\lambda(x)|\to0\quad\mbox{as }r\searrow 0,
$$
$$
\sup_{y\in B_-(x,r)}r^{-1}|v(y)-v(x)-y\cdot
\mu(x)|\to0\quad\mbox{as }r\searrow 0,
$$ 
where 
$$
B_+(x,r)=\{y\in B(x,r)\,:\,(y-x)\cdot \nu(x)>0\} ,
$$ 
$$
B_-(x,r)=\{y\in B(x,r)\,:\,(y-x)\cdot \nu(x)<0\} .
$$ }
\end{lem}
Observe that by our definitions the inclusion
$$
Z_v\supset\{x\in G_v : \nabla v(x)=0\}\cup\{x\in K_v:\mu(x)=0\mbox{ or }\lambda(x)=0\}
$$
holds.

\begin{lem}[\cite{Am}]
\label{lem8} {\sl For any Lebesgue measurable set $F\subset \R^2$
with $\H^1(\partial^M F)<\infty$ there is a finite or countable
family  $\{F_i\}_{i\in I}$ and a set $T\subset\R^2$ with the
following properties:

(i) $F_i$ are measurable sets, $\mathcal L^2 (F_i)>0$,
$\H^1(\partial^M F_i)<\infty$;

(ii) $F_i\cap F_j=\emptyset$ \ for $i\ne j$;

(iii) $(\partial^M F_i)\cap(\partial^M F_j)=\emptyset\ (\mathrm{mod}
\H^1)$ \ for $i\ne j$.

(iv) $\partial^M F=\bigcup\limits_{i\in I}\partial^M F_i\ (\mathrm{mod}
\H^1)$,\  so in particular, \linebreak$\H^1(\partial^M
F)=\sum\limits_{i\in I}\H^1(\partial^M F_i)$.

(v) $\H^1\biggl(\Int_M F\setminus\biggl(\bigcup\limits_{i\in
I}\Int_MF_i\biggr)\biggr)=0$.

(vi)  $\H^1(T)=0$.

(vii) For any set $L$ with $\H^1(L)=0$ and for any $x,y\in
\Int_MF_i\setminus (T\cup L)$ and $\delta>0$ there exists a
rectifiable curve $\Gamma\subset (\Int_M F_i)\setminus (T\cup L)$
joining $x$ to $y$ so that
$$
\H^1(\Gamma)\le|x-y|+\H^1(\partial^M F_i)+\delta.
$$
}
\end{lem}

\begin{proof}
See Proposition 3,  Theorems 1 and 8 (together with the
subsequent remark) from \cite{Am}. 
\end{proof}

\begin{lem}
\label{lem8.1} {\sl If the set $F$ in Lemma~\ref{lem8} is bounded, then we can 
reformulate the property (vii) in the following way:

(vii') for any set $L$ with $\H^1(L)=0$ and for any $x,y\in
(\Int_MF_i)\setminus (T\cup L)$ and $\delta>0$ there exists a
rectifiable curve $\Gamma\subset (\Int_M F_i)\setminus (T\cup L)$
joining $x$ to $y$ so that
$$
\H^1(\Gamma)\le 2\H^1(\partial^M F_i)+\delta.
$$
}
\end{lem}

\begin{proof}
See \cite[Lemma~4.2]{PZ}. 
\end{proof}

\begin{lem}
\label{lem8.2} {\sl Suppose
$\H^1(\partial^ME_\varepsilon)<\infty$. Let $E^i_{\varepsilon}$ be
the sets from Lemmas~\ref{lem8}-\ref{lem8.1} applying to
$F=E_\varepsilon$. Then
$\diam(v(\Cl_ME^i_{\varepsilon}))\le2\varepsilon
\H^1(\partial^ME^i_{\varepsilon})$.}
\end{lem}

\begin{proof}
In property (vii') of Lemma~\ref{lem8.1} put $L=A_v$,
where $A_v$ is defined in Lemma~\ref{lem3.1.1}. Then
the restriction $v|_\Gamma$ is $\varepsilon$--Lipschitz. 
\end{proof}

\begin{lem}
\label{lem8.3} {\sl For any $\varepsilon>0$ the inequality
$\H^1(v(\Cl_ME_{\varepsilon}))\le2\varepsilon
\H^1(\partial^ME_{\varepsilon})$ holds.}
\end{lem}

\begin{proof}
Suppose $\H^1(\partial^ME_\varepsilon)<\infty$. From
properties (iv)-(v) of Lemma~\ref{lem8} we have $\Cl_M
E_\varepsilon=\bigcup\limits_{i\in I}\Cl_M E^i_\varepsilon\ (\mod
\H^1)$. So from Corollary~\ref{Thh3.1} we obtain
$$
\H^1(v(\Cl_M E_{\varepsilon}))\le\sum\limits_{i\in I}\H^1(v(\Cl_M
E^i_\varepsilon))\le2\varepsilon\sum\limits_{i\in
I}\H^1(\partial^ME^i_{\varepsilon})=
2\varepsilon\H^1(\partial^ME_{\varepsilon}) ,
$$ 
where the last equality follows from property~(iv) of Lemma~\ref{lem8}. 
\end{proof}

\begin{cor}
\label{lem9} {\sl For any $\varepsilon>0$  the estimate
\begin{equation} \label{8}
\H^1(v(\Cl_M E_\varepsilon))\le
2\varepsilon\bigl[\H^1(\Omega\cap\partial^ME_\varepsilon)+\H^1(\partial\Omega)\bigl]
\end{equation}
holds.}
\end{cor}

\begin{cor}
\label{cor10} {\sl The convergence
\begin{equation} \label{9}
\H^1(v(\Cl_M E_\varepsilon))\to0\quad\mbox{as }\varepsilon\to0+
\end{equation}
holds.}
\end{cor}

\begin{proof}
It follows from Lemma~\ref{lem9} and the Coarea
formula (see also the proof of Preposition~4.3 in~\cite{PZ}).
\end{proof}

Obviously the last corollary, together with Lemma~\ref{lem3.1.1}
and Corollary~\ref{Thh3.1}, imply the statement of
Theorem~\ref{Th1}.

\section{Application to the level sets of $\WW^{2,1}$ functions}

By {\em a cycle} we mean a set which is homeomorphic to the unit
circle ${\mathbb S}^1 \subset\R^2$. Now the purpose of the section is to
prove the following result.

\begin{ttt}
\label{Th2.1} {\sl Suppose $v\in \WW^{2,1}(\R^2)$. Then for almost
all $y\in\R$ the preimage $v^{-1}(y)$ is a finite disjoint family
of $\CC^1$-cycles $S_j$, $j=1,\dots,N(y)$. Moreover, the tangent
vector to each $S_j$ is an absolutely continuous function. }
\end{ttt}

Invoking extension theorems for Sobolev spaces (see, for 
example,~\cite{M}), we obtain the following:

\begin{cor}
\label{cor2.2} {\sl Suppose $\Omega\subset\R^2$ is a bounded
domain with a Lipschitz  boundary and $v\in \WW^{2,1}(\Omega)$. Then
for almost all $y\in\R$ the preimage $v^{-1}(y)$ is a finite
disjoint family of $\CC^1$-curves $\Gamma_j$, $j=1,\dots,N(y)$. Each
$\Gamma_j$ is a cycle or it is a simple arc with endpoints on
$\partial\Omega$ (in case of the latter, $\Gamma_j$ is transversal to
$\partial\Omega$). Moreover, the tangent vector to each $\Gamma_j$
is an absolutely continuous function. }
\end{cor}
Fix a function $v\in \WW^{2,1}(\R^2)$.

\begin{lem}
\label{lem1} {\sl For any $\alpha\in(0,1)$, a ball
$B(x,r)\subset\R^2$ and for any Lebesgue measurable set
$E\subset B(x,r)$ satisfying $\frac{\mathcal L^2(E)}{\mathcal L^2
(B(x,r))}\ge\alpha$ the estimate
\begin{equation}
\label{1} \sup\limits_{y\in
B(x,r)}\biggl|v(y)-v(x)-y\cdot\dashint\limits_E\nabla v(z)\,\dd z\biggr|\le
c_\alpha \|D^2v\|(B(x,r))
\end{equation}
holds, where $c_\alpha$ depends on $\alpha$ only.}
\end{lem}

\begin{proof}
Because of coordinate invariance it is sufficient to
prove the estimate for the case $\Omega=B(0,1)=B(x,r)$. By results
of \cite{M} for any $u\in \WW^{2,1}(\Omega)$ the estimate
\begin{equation}
\label{1.1} \sup\limits_{y\in \Omega}|u(y)|\le
c(p)\bigl(p(u)+\|D^2u\|(\Omega)\bigr),
\end{equation}
holds, where $p(\cdot)$ is a continuous seminorm in $\WW^{2,1}(\Omega)$
such that $p(g)=0\Leftrightarrow g=0$ for all first-order
polynomials~$g$. Take $p_\alpha(u)=|u(0)|+\inf\limits_{E\subset
\Omega,\ \L^2(E)\ge\alpha }\biggl|\dashint\limits_E
\nabla u(z)\,\dd z\biggr|$. It is easy to check that $p_\alpha$ satisfies
the above conditions. Fix a measurable set $E\subset\Omega$ with
$\L^2(E)\ge\alpha$ and take
$u(y)=v(y)-v(0)-y\cdot\dashint\limits_E\nabla v(z)\,\dd z$. Then
$p_\alpha(u)=0$ and the inequality~(\ref{1.1}) turns to the
estimate~(\ref{1}). 
\end{proof}

\noindent
For functions $v\in \WW^{2,1}(\R^2)$ the set $K_v$ from Lemma~\ref{lem3.1.1} is 
empty (see the proofs in~\cite{Dor}), so we have the following result.

\begin{lem}[see also Theorem~1 in \cite{EG}, \S{4.8}]
\label{Thh2.1.1} {\sl We can choose the representative of $\nabla v$
such that there exists a set $A_v\subset\R^2$ with with the
following properties:

(i) $\H^1(A_v)=0$;

(ii) for all $x\in \R^2\setminus A_v$
$$
\lim_{r\searrow 0}\dashint_{B(x,r)}|\nabla v(z)-\nabla v(x)|^2\,\dd z=0,
$$
$$
\sup_{y\in B(x,r)}r^{-1}|v(y)-v(x)-y\cdot
\nabla v(x)|\to0\quad\mbox{as }r\searrow 0
$$ 
(i.e., $v$ is differentiable at $x$);

(iii) for any $\varepsilon>0$ there exists an open set
$U\subset\R^2$  such that $\Capp_1(U)<\varepsilon$, $A_v\subset
U$, and $\nabla v$ is continuous on $\R^2\setminus U$. }
\end{lem}
Further we fix the above representative of $\nabla v$. Here (see, for example, 
\cite[\S{4.8}]{EG}) $\Capp_1$ denotes the 1-capacity defined for any $E\subset\R^2$
as
$$
\Capp_1(E)=\inf \Bigl\{ \|\nabla f\|_{L^1}: f\in L^2(\R^2),\ Df\in L^1(\R^2), f\ge1 
\mbox{ in an open neighborhood  of } E \Bigr\} .
$$
The 1-capacity has the following simple description.

\begin{lem}[see the proof of Theorem~3 in \cite{EG}, \S5.6.3]
\label{lem2.2}{\sl There is a constant $C_0>0$ such that for any
set $E\subset\R^2$ the inequalities
$$
\frac1{C_0}\H^1_\infty(E)\le\Capp_1(E)\le C_0\H^1_\infty(E)
$$
hold.}
\end{lem}

\begin{lem}
\label{lem2.3} {\sl For any $\varepsilon>0$ there exists an open
set $U\subset\R^2$ and a function $g\in C^1(\R^2)$ such that
$\Capp_1(U)<\varepsilon$, $A_v\subset U$ and $v|_{\R^2\setminus
U}=g|_{\R^2\setminus U}$, $\nabla v|_{\R^2\setminus
U}=\nabla g|_{\R^2\setminus U}$. }
\end{lem}

\begin{proof}
Denote
$$
A_{\delta,\rho}=\{x\in \R^n:\exists r\in(0,\rho]\ \mbox{ so }
\frac1r\|D^2v\|(B(x,r))\ge\delta\}.
$$
Using Vitali's covering theorem (see~\cite{EG}) and that $\| D^{2}v \|$
is absolutely continuous with respect to $\L^2$ (recall that $v$ is $\WW^{2,1}$)
it is easy to prove that for each fixed  $\delta > 0$,
\begin{equation}
\label{l2.1}
\Capp_1(A_{\delta,\rho})\to0\mbox{ as }\rho\searrow 0.
\end{equation}
So we can choose a sequence $\rho_j>0$ such that 
\begin{equation}
\label{l2.2} \Capp_1(A_{\frac1j,\rho_j})\le\frac1{2^j}
\end{equation}
holds. Denoting
$$
A_k=\bigcup\limits_{j\ge k}A_{\frac1j,\rho_j},
$$
we have
\begin{equation}
\label{l2.3} \Capp_1(A_{k})\le\frac1{2^{k-1}};
\end{equation}
\begin{equation}
\label{l2.4} \forall k\in\N\ \forall \alpha>0\ \exists
r_{k,\alpha}>0\ \forall x\in\R^2\setminus A_k\ \forall
r\in(0,r_{k,\alpha})\quad\frac1r\|D^2v\|(B(x,r))<\alpha.
\end{equation}
It follows from the proof of Theorem~1 in \cite[\S4.8]{EG} that there exists 
a sequence of mappings $f_i\in \CC_0^\infty(\R^2,\R^2)$ such
that for the sets
\begin{equation}
\label{l2.6} B_i=\{x\in\R^n: \exists r>0\ \
\dashint_{B(x,r)}|\nabla v(y)-f_i(y)|\, \dd y>\frac1{2^i}\},
\end{equation}
$$
F_k=A_v\cap\left(\bigcup\limits_{j=k}^\infty B_j\right)
$$
we have
$$
\Capp_1 F_k\to 0\quad\mbox{ as } \quad k\to\infty,
$$
and
\begin{equation}
\label{l2.5} \forall x\in\R^2\setminus F_k\ \forall i\ge
k\quad|f_i(x)-\nabla v(x)|\le\frac1{2^i} .
\end{equation}
Take a sequence of open sets $U_k\supset F_k\cup A_k$ such that
\begin{equation}
\label{l2.7} \Capp_1 U_k\to 0\quad\mbox{ as } \quad k\to\infty.
\end{equation}
Then from above formulas~(\ref{l2.4})--(\ref{l2.5}) and
Lemma~\ref{lem1} we obtain that there exist a function
$\omega\colon (0,+\infty)\to(0,+\infty)$ such that $\omega(\delta)\to0$
as $\delta\searrow 0$ and for all $k\in\N$ and for any pair $x,y\in
\R^2\setminus U_k$ the estimates
$$
|v(x)-v(y)|\le\omega(|x-y|),
$$
$$
|\nabla v(x)-\nabla v(y)|\le\omega(|x-y|),
$$
$$
|v(y)-v(x)-(y-x)\cdot \nabla v(x)|\le\omega(|x-y|)|x-y|
$$
hold. Then the assertion of Lemma~\ref{lem2.3} follows from the
last estimates, the convergence~(\ref{l2.7}), and from the
classical Whitney extension theorem (see, for example,
\cite[Theorem~1 of \S6.5]{EG}). 
\end{proof}
Using Theorems \ref{Thh3.3}, \ref{Th1} and Lemma~\ref{lem2.2} we
can reformulate the last lemma in the following way.

\begin{cor}
\label{cor2.4} {\sl For any $\varepsilon>0$ there exist an open
set $V\subset\R$ and a function $g\in \CC^1(\R^2)$ such that
$\H^1(V)<\varepsilon$, $v(A_v)\subset V$ and
$v|_{v^{-1}(\R\setminus V)}=g|_{v^{-1}(\R\setminus V)}$,
$\nabla v|_{v^{-1}(\R\setminus V)}=\nabla g|_{v^{-1}(\R\setminus V)}\ne0$. }
\end{cor}
The last corollary and Lemma~\ref{Thh2.1.1} easily imply the
statement of Theorem~\ref{Th1}.

\section{Application to the level sets of $\BV^2$ functions}

The main goal of this section is to prove the following result.

\begin{ttt}
\label{Th3.1} {\sl Suppose $v\in \BV_2(\R^2)$. Then for almost all
$y\in\R$ the preimage $v^{-1}(y)\cap\Omega$ is a finite disjoint
family of  cycles $S_j$, $j=1,\dots,N(y)$.
Moreover, the variation
of the tangent vector to each $S_j$
(i.e., the integral curvature of $\Gamma_j$) is finite.}
\end{ttt}

\begin{cor}
\label{cor3.2} {\sl Suppose $\Omega\subset\R^2$ is a bounded
domain with a Lipschitz  boundary and $v\in \BV^{2}(\Omega)$. Then
for almost all $y\in\R$ the preimage $v^{-1}(y)$ is a finite
disjoint family of Lipschitz curves $\Gamma_j$, $j=1,\dots,N(y)$.
Each $\Gamma_j$ is a cycle or it is a simple arc with endpoints on
$\partial\Omega$ (in the last case $\Gamma_j$ is transversal to
$\partial\Omega$). Moreover, the variation of the tangent vector
to $\Gamma_j$ (i.e., the integral curvature of $\Gamma_j$) is
finite. }
\end{cor}
Curves of this kind are called {\em curves of finite turn} and they
have been systematically studied in~\cite{AR} and \cite{Re1}.

\noindent
Fix a function $v\in \BV^{2}(\R^2)$.
Let $A_v,\ K_v,\ \mu(x),\ \lambda(x),\ \nu(x)$ be objects defined in 
Lemma~\ref{lem3.1.1}.

\begin{lem}
\label{cor3.1.2} {\sl For almost all $y\in v(\R^2)$ the following
assertions are true:

(i) $v^{-1}(y)\cap A_v=\emptyset$;

(ii) for all $x\in v^{-1}(y)$ \ $\lambda(x)\ne0\ne\mu(x)$;

(iii) for all $x\in v^{-1}(y)\cap K_v$ both vectors $\lambda(x)$,
$\mu(x)$ are not parallel to $\nu(x)$;

(iv) the intersection $v^{-1}(y)\cap K_v$ is at most countable;

(v) $\H^1(v^{-1}(y))<\infty$. }
\end{lem}

\begin{proof}
(i) follows from Theorem~\ref{Thh3.1}.

(ii) follows from Theorem~\ref{Th1}.

(iii) follows from the classical one dimension version of the Sard
theorem applied to the restriction $v|_{L_i}$ (see the assertion
(ii), (iv) of Lemma~\ref{lem3.1.1} );

(iv) follows from (iii);

(v) follows from the Coarea formula. 
\end{proof}
By {\it connectedness\/} (without additional terms) we mean connectedness in 
the sense of general topology.

\begin{lem}[see, for example, Lemma 2.2 in \cite{K07}]
\label{lemt3} {\sl Let $\Omega\subset\R^2$ be a domain that is homeomorphic to the 
unit disc and let $G\subset\Omega$ be a subdomain of~$\Omega$. Then for each connected
component~$\Omega_i$ of the open set $\Omega\setminus\Cl G$ the
intersection~$\Omega\cap\partial\Omega_i$ is ~connected.} 
\end{lem}

\begin{lem}[see, for example, \cite{Am}]
\label{lem3.1.3} {\sl Suppose $K$ is a compact connected set in
$\R^2$ and $\H^1(K)<\infty$. Then $K$ is arcwise connected. }
\end{lem}
By {\it arc} we mean a
set which is homeomorphic to an interval of the straight line.

\begin{lem}
\label{lem3.1.4} {\sl For any $y\in\R$ satisfying (i)--(v) of
Lemma~\ref{cor3.1.2}, for any $x\in v^{-1}(y)$, and for all
sufficiently small $r>0$ the connected component $K\ni x$ of the
set $B(x,r)\cap v^{-1}(y)$ contains an arc $J\ni x$ with endpoints
on $\partial B(x,r)$. Moreover, the set $J\setminus\{x\}$ intersects
two connected components of the set $B(x,r)\cap
v^{-1}(y)\setminus\{x\}$. }
\end{lem}

\begin{proof} 
We may assume without loss of generality that $x=0$, $v(x)=0$ and the vector
$\nu(x)$ (from Lemmas~\ref{lem3.1.1}, \ref{cor3.1.2}) is
vertical: $\nu(x)=(0,1)$. Let $L$ be the intersection of the open 
ball $B(0,r)$ with the horizontal axis:
$L=\{(t,0):t\in(-r,r)\}$. Denote by $A,C$ the endpoints of the segment~$L$:
$A=(r,0)$, $C=(-r,0)$. If $r>0$ is sufficiently small, then
by the differentiability properties recorded in
Lemmas~\ref{lem3.1.1}, \ref{cor3.1.2} we infer that the function $v$ is strictly 
monotone on $L$.
For definiteness assume that $v(t,0)>0$ for $t\in(0,r]$ and $v(t,0)<0$ for $t\in[-r,0)$. 
In particular, $v(A)>0>v(C)$.
Denote $\Omega_+=\{(t,s)\in B(0,r):s>0\}$, $\Omega_-=\{(t,s)\in B(0,r):s<0\}$.
Denote by $G$ the connected component of the open set $\{z\in \Omega_+:v(z)>0\}$ such that
$A\in\partial G$. Denote by $\Omega_1$ the connected component of the open set
$\Omega_+\setminus\Cl G$ such that $C\in \partial\Omega_1$. 
Put $K_+=\Cl (\Omega_+\cap\partial\Omega_1)$.
Obviously $0\in K_+$, \ $v\equiv 0$ on $K_+$, 
and \linebreak $K_+\cap(\partial\Omega_+)\setminus\Cl\Omega_-\ne\emptyset$.
Let $D_+\in K_+\cap(\partial\Omega_+)\setminus\Cl\Omega_-$.
By Lemma~\ref{lemt3} $K_+$ is a compact connected set, and by (v) of Lemma~\ref{cor3.1.2}
$\H^{1}(K_+)<\infty$. Then by Lemma~\ref{lem3.1.3} there exists an arc 
$J_+\subset K_+$ joining $0$ to $D_+$. 
Because $L\cap v^{-1}(0)=\{0\}$ we have equality $J_+\cap\Cl \Omega_-=\{0\}$.
Analoguously, there exists a point $D_-\in(\partial\Omega_-)\setminus\Cl\Omega_+$ and 
an arc $J_-\subset \Cl(\Omega_-\cap v^{-1}(0))$ joining $0$ to $D_-$ so that
$J_-\cap\Cl \Omega_+=\{0\}$. Now $J=J_+\cup J_-$ is the required arc.
\end{proof}

\begin{lem}
\label{lem3.1.5} {\sl For any $y\in\R$ satisfying (i)--(v) of
Lemma~\ref{cor3.1.2} and for any connected component $C$ of
$v^{-1}(y)$ there exists a cycle $S\subset C$. Moreover, if there
is only one cycle $S\subset C$, then $S=C$. }
\end{lem}

\begin{proof}
Let $J_1$ be a maximal {\it open} arc  (i.e.,
$J_1$ is homeomorphic to an {\it open} interval of $\R$) in $C$. Such an arc exists
by the previous Lemma~\ref{lem3.1.4}. By (v) of
Lemma~\ref{cor3.1.2} the inequality $\H^1(J_1)<\infty$ holds.
So the arc $J_1$ has endpoints, denote them by $x$, $y$. If $x=y$, then there is
nothing to prove. The same applies for the case $x\in J_1$. If $x\ne y$
and $x\notin J_1$ we can continue the arc $J_1$ through $x$ by
Lemma~\ref{lem3.1.4}. This contradiction establishes the
existence of a cycle $S\subset C$.

To prove the second statement suppose that $z\in C\setminus S$.
Take a maximal arc $J_2$ in $C$ containing~$z$. By the above arguments
this arc generates a cycle $S_2\ne S$, $S_2\subset C$. 
\end{proof}

\begin{cor}
\label{lem3.1.6} {\sl There exists at most countable set $Z\subset
\R$ such that for any $y\in\R\setminus Z$ satisfying (i)--(v) of
Lemma~\ref{cor3.1.2} any connected component $C$ of
$v^{-1}(y)$ is a cycle. }
\end{cor}

\begin{proof}
Suppose $y\in\R$ satisfies (i)--(v) of
Lemma~\ref{cor3.1.2} and a connected component $C$ of
$v^{-1}(y)$ is not a cycle. Then by Lemma~\ref{lem3.1.5} the set
$\R^2\setminus C$ has more than two connected  components. By
results of \cite{Kr} this is possible only for at most countable
many values of~$y$. 
\end{proof}

We need the following classical estimate and its corollary:

\begin{lem}[see, for example, Lemma~1 of \S4.8 in \cite{EG}]
\label{lem3.4} {\sl There exists the constant $C_5>0$ such that
the estimate
$$
\Capp_1(\{x\in\R^2\,:\,\exists r>0\ \dashint_{B(x,r)}|\nabla v(y)|\,\dd y\ge
\delta\})\le{C_5}\frac1\delta\|D^2v\|(\R^2)
$$ 
holds.}
\end{lem}

\begin{cor}
\label{lem3.5} {\sl The estimate
$$
\Capp_1(\{x\in G_v\,:\,|\nabla v(x)|>\delta\})\le{C_5}\frac1\delta\|D^2v\|(\R^2)
$$ 
holds.}
\end{cor}

\begin{lem}
\label{lem3.6} {\sl For any $\varepsilon>0$ there exists a compact
set $F_\varepsilon\subset v(\R^2)$ and constants
$\delta_1,\delta_2>0$ such that $\H^1(v(\R^2)\setminus
F_\varepsilon)<\varepsilon$ and for all $y\in F_\varepsilon$ the
preimage $v^{-1}(y)$ satisfies the properties (i)-(v) from the
Lemma~\ref{cor3.1.2} and the following additional conditions:

(vi) for all $x\in v^{-1}(y)\cap G_v$ the estimates
$\delta_1>|\nabla v(x)|>\delta_2$ hold;

(vii) each connected component of the set $v^{-1}(y)$ is a cycle.}
\end{lem}

\begin{proof}
(vi) follows from Theorem~\ref{Thh3.3},
Lemma~\ref{lem2.2} and Corollaries~\ref{cor10}, \ref{lem3.5}.
(vii) follows from Lemma~\ref{lem3.1.6}. 
\end{proof}

\begin{proof}[Proof of Theorem~\ref{Th3.1}] 
Fix an arbitrary $\varepsilon>0$
and take the set $F_\varepsilon$ from Lemma~\ref{lem3.6}. From
the above results we have that
$$
\forall y\in F_\varepsilon\quad v^{-1}(y)=\bigcup\limits_{j=1}^{N(y)}S_j(y),
$$
where $S_j(y)$ are cycles and $N(y)\in\mathbb N\cup\{+\infty\}$.

Take a sequence of functions $v_{i}\in C^\infty(\R^2)$ that
approximates $v$ as usual. In particular,
\begin{eqnarray}
\label{t3.1} \forall x\in G_v \quad \nabla v_{i}(x)\to \nabla v(x);\\
\label{t3.2}
\|D^2v_{i}\|(\R^2)=\int\limits_{\R^2}|D^2v_i(x)|\,\dd x\le2\|D^2v\|(\R^2)
\end{eqnarray}
By the coarea formula
\begin{equation}
\label{t3.3}
\int\limits_{v^{-1}(F_\varepsilon)}|\nabla v(x)|\cdot|D^2v_{i}(x)|\,\dd x=
\int\limits_{F_\varepsilon}\sum\limits_{j=1}^{N(y)}\int\limits_{S_j(y)}|D^2v_{i}(x)|
\,\dd \H^1\,\dd y\le 2\delta_1\|D^2v\|(\R^2),
\end{equation}
where the last estimate follows from condition (vi) of
Lemma~\ref{lem3.6}. Consequently there exists a constant $C_7$ such that
\begin{equation}
\label{t3.3.1}
\int\limits_{F_\varepsilon}\sum\limits_{j=1}^{N(y)}\Var(\nabla v_{i},S_j(y))\,\dd 
y\le C_7,
\end{equation}
where $\Var(\nabla v_{i},S_j(y))$ is the variation of $\nabla v_{i}$ on
$S_j(y)$.

\noindent
From (\ref{t3.1}) and the properties (i), (iv) of
Lemma~\ref{cor3.1.2} it is easy to deduce that
\begin{equation}
\label{t3.4}
\Var(\nabla v,S_j(y))\le\liminf\limits_{{i}\to\infty}\Var(\nabla v_{i},S_j(y)),
\end{equation}
consequently,
\begin{equation}
\label{t3.5} \sum\limits_{j=1}^{N(y)}\Var(\nabla v,S_j(y))\le
\liminf\limits_{{i}\to\infty}\sum\limits_{j=1}^{N(y)}\Var(\nabla v_{i},S_j(y)).
\end{equation}
Then by Fatou's lemma
\begin{equation}
\label{t3.6}
\int\limits_{F_\varepsilon}\sum\limits_{j=1}^{N(y)}\Var(\nabla v,S_j(y))\,\dd y
\le\liminf\limits_{{i}\to\infty} \int\limits_{F_\varepsilon}\sum\limits_{j=1}^{N(y)}
\Var(\nabla v_{i},S_j(y))\,\dd y\le C_7.
\end{equation}
Let ${\mathbf\tau}$ denote the tangent vector to $S_j(y)$. By straightforward
geometric considerations we have
\begin{equation}
\label{t3.7} 2\pi\le\Var({\mathbf\tau},S_j(y))\le
\frac{\delta_1}{(\delta_2)^2}\Var(\nabla v,S_j(y))
\end{equation}
From the last two formulas we deduce that $N(y)<\infty$ and
$\sum\limits_{j=1}^{N(y)}\Var({\mathbf\tau},S_j(y))<\infty$ for
almost all $y\in F_\varepsilon$. 
\end{proof}

\noindent
School of Mathematics, Institute for Advanced Study, Einstein Drive, Princeton N.J.~08540, USA\\
e-mail: {\it bourgain@math.ias.edu}
\bigskip

\noindent
Sobolev Institute of mathematics, Acad.~Koptyug pr., 4, Novosibirsk 630090, Russia\\ 
e-mail: {\it korob@math.nsc.ru}
\bigskip

\noindent
Mathematical Institute, University of Oxford, 24--29 St.~Giles', Oxford OX1 4AU, England\\
e-mail: {\it kristens@maths.ox.ac.uk}

\end{document}